\documentclass[11pt]{amsart}
\usepackage{xcolor}
\usepackage{amsmath}
\usepackage{amssymb}
\usepackage{tikz}
\usepackage{relsize}
\usepackage{blkarray}
\usepackage{mathrsfs}
\usepackage[scr=rsfs,cal=boondox]{mathalfa}
\usepackage{hyperref}

\hypersetup{colorlinks=true,linkcolor=blue,citecolor=red}

\newtheorem{Definition}{Definition}[section]
\newtheorem{Theorem}[Definition]{Theorem}
\newtheorem{Lemma}[Definition]{Lemma}
\newtheorem{Proposition}[Definition]{Proposition}
\newtheorem{Corollary}{Corollary}[section]

\newtheorem{Remark}[Definition]{Remark}

\newcommand{\Z}{\mathbb Z}%
\allowdisplaybreaks[1]

\begin{document}
\baselineskip16pt

\author[Sumit Kumar Rano]{Sumit Kumar Rano}

\address{Sumit Kumar Rano \endgraf Stat-Math Unit,	\endgraf Indian Statistical Institute,	\endgraf 203 B. T. Road, Kolkata 700108, India.} \email{sumitrano1992@gmail.com}

\title[Eigenfunctions of the Laplacian satisfying Hardy-type estimates on hom. trees]
{Eigenfunctions of the Laplacian satisfying Hardy-type estimates on homogeneous trees}

\begin{abstract}
This work deals with the characterization of eigenfunctions of the Laplacian $\mathcal{L}$ on a homogeneous tree $\mathcal{X}$, which satisfy certain growth conditions. More precisely, we shall prove that the Poisson transform on $\mathcal{X}$ provides an one-to-one correspondence between the subspace of all Hardy-type eigenfunctions of $\mathcal{L}$ on $\mathcal{X}$ and the Lebesgue spaces (possibly the set of all complex measures) on the boundary of $\mathcal{X}$.
\end{abstract}

\maketitle

\section{Introduction}
The study of eigenfunctions of the Laplacian on various spaces and its representation as the Poisson integral of functions defined on their respective boundaries have always been a central theme of investigation in the field of harmonic analysis and partial differential equation. When the underlying space is the open unit disc $\mathbb{D}$ with the boundary $\partial\mathbb{D}$, such study is mostly concentrated around harmonic functions. An important prototypical result in this direction says that the Poisson integral
$$Pf(se^{i\theta})=\frac{1}{2\pi}~\int\limits_{-\pi}^{\pi}p_{s}(\theta-\psi)f(\psi)d\psi,\quad\text{where }0\leq s<1\text{ and }-\pi\leq\theta\leq\pi,$$
provides an one-to-one correspondence between the Lebesgue spaces $L^{r}(\partial\mathbb{D})$ (respectively the set of all complex measures on $\partial \mathbb{D}$ when $r=1$) and the subspace of all harmonic functions $u$ on $\mathbb{D}$ which satisfies the Hardy-type growth conditions (see \cite{R2})
$$\|u\|_{H_{r}(\mathbb{D})}:=
\begin{cases}
\sup\limits_{0\leq s<1}\left(~\frac{1}{2\pi}~\int\limits_{-\pi}^{\pi}|u(se^{i\theta})|^{r}d\theta~\right)^{1/r}<\infty,&\quad\text{when }1\leq r<\infty,\\
\sup\limits_{z\in\mathbb{D}} |u(z)|<\infty,&\quad\text{when } r=\infty.
\end{cases}
$$

In the context of a homogeneous tree $\mathcal{X}$, the Poisson boundary representation of harmonic functions with respect to the Laplacian $\mathcal{L}$ was given by Cartier \cite{C}. Among other things, the author proved that $\mathcal{L}u=0$ if and only if $u$ is the Poisson integral of a martingale defined on the boundary $\Omega$ of $\mathcal{X}$ (see \cite[Proposition A.4]{C}). Thereafter, Cohen et al. \cite{CCS} introduced the notion of {\em harmonic Hardy spaces} on $\mathcal{X}$ and a characterization similar to that on $\mathbb{D}$ was obtained on $\mathcal{X}$. We also refer to the articles \cite{BP,KPT,T} to find more about harmonic Hardy spaces on $\mathcal{X}$.

On the other hand, soon after the discovery by Cartier, this study went far beyond the realms of only harmonic functions of $\mathcal{L}$ on $\mathcal{X}$. This is due to a highly non-Euclidean fact that there is an abundance of eigenfunctions of $\mathcal{L}$ on $\mathcal{X}$ with complex eigenvalues which are well behaved in terms of certain size conditions. For $z\in\mathbb{C}$, the Poisson transform $\mathcal{P}_{z}F$ of a nice function $F$ on $\Omega$ is defined by (see \cite[Chapter 4, Formula (1) at Page 53]{FTP2})
$$\mathcal{P}_{z}F(x)=\int\limits_{\Omega}p^{1/2+iz}(x,\omega)F(\omega)d\nu(\omega),\quad\text{for all }x\in\mathcal{X},$$
where $p(x,\omega)$ denotes the Poisson kernel on $\mathcal{X}$. Fig\`{a}-Talamanca et al. \cite{FTP1} proved that $\mathcal{L}(\mathcal{P}_{z}F)=\gamma(z)\mathcal{P}_{z}F$, where $\gamma$ is an entire function as in (\ref{gammaz}) below. However, not all eigenfunctions of $\mathcal{L}$ with eigenvalue $\gamma(z)$ are of the above form.

\begin{Theorem}[{{\cite[Theorem A]{MZ}}}]\label{MZ2}
Let $z\in\mathbb{C}$ and $\tau=2\pi/\log q$. If $z\neq (k\tau+i)/2$, where $k\in\mathbb{Z}$ then the Poisson transform $\mathcal{P}_{z}$ is a bijective map between the space of all $\gamma(z)$-eigenfunctions of $\mathcal{L}$ on $\mathcal{X}$ and the space of all martingales defined on the boundary $\Omega$.
\end{Theorem}

Our aim in this paper is to look for a refinement of Theorem \ref{MZ2} by imposing the Hardy-type size conditions on the eigenfunctions of $\mathcal{L}$ on $\mathcal{X}$. These non-Euclidean size estimates arise naturally due to the behaviour of the Poisson transforms $\mathcal{P}_{z}$ which are matrix coefficients of certain representations $\pi_{z}$ of the group of automorphisms $G$ of $\mathcal{X}$, acting on function spaces over $\Omega$. We refer to \cite{CMS,FTP1,FTP2,KP,KR,MZ} for details. Appearance of these estimates is natural since they are intimately linked to the Kunze-Stein phenomenon \cite{CMS,V}, a key non-Euclidean aspect of homogeneous trees.

Before stating our main results, we need to introduce a few more notations. For details, we refer to Section \ref{Section 2}. The number $2\pi/\log q$ is denoted by $\tau$. For any $p\in(1,\infty)$, let $p^{\prime}$ denote the conjugate index $p/(p-1)$ and $\delta_{p}=-\delta_{p^{\prime}}=1/p-1/2$. We assume that $p^{\prime}=\infty$ when $p=1$ and vice-versa. Thus $\delta_{1}=-\delta_{\infty}=1/2$. For $1\leq p\leq 2$ and $1\leq r\leq \infty$, $\mathcal{H}^{r}_{p}(\mathcal{X})$ denotes the Hardy-type spaces on $\mathcal{X}$ which is defined as the collection of all complex-valued functions $u$ on $\mathcal{X}$ that satisfies $\|u\|_{\mathcal{H}^{r}_{p}(\mathcal{X})}<\infty$ (see (\ref{hardy1}) and (\ref{hardy2}) below for the definition).

\begin{Theorem}\label{hardypcase}
Let $1\leq p\leq 2$ and $u$ be a complex-valued function on $\mathcal{X}$. Suppose that $z=\alpha+i\delta_{p^{\prime}}$, where $\alpha$ satisfies either of the following conditions:
\begin{enumerate}
\item[(i)] If $1\leq p<2$, $\alpha\in\mathbb{R}$.
\item[(ii)] If $p=2$, $\alpha\in(\tau/2)\mathbb{Z}$.
\end{enumerate}
Then $u=\mathcal{P}_{z}F$ for some $F\in L^{r}(\Omega)$ if and only if $u\in \mathcal{H}^{r}_{p}(\mathcal{X})$ for $1<r\leq\infty$ and $\mathcal{L}u=\gamma(z)u$. When $r=1$, the above result holds true with $u=\mathcal{P}_{z}\mu$ for some complex measure $\mu$ on the boundary $\Omega$. Moreover, there exists a positive constant $C_{p}$ depending only on $p$ such that
$$C_{p}\|F\|_{L^{r}(\Omega)}\leq\|\mathcal{P}_{z}F\|_{\mathcal{H}^{r}_{p}(\mathcal{X})}\leq\|F\|_{L^{r}(\Omega)},\quad\text{for all }F\in L^{r}(\Omega),\text{ for all }1\leq r\leq \infty.$$
\end{Theorem}

In particular, when $p=1$ and $\alpha\in\tau\mathbb{Z}$ in Theorem \ref{hardypcase}, we get back the characterization results \cite[Theorems 2.3 and 2.4]{CCS} proved for harmonic functions of $\mathcal{L}$ on $\mathcal{X}$ (see also \cite[Proposition 10]{KPT}). However, Theorem \ref{hardypcase} does not include those real numbers $\gamma(z)$ that lie in the interior of the $L^{2}$-spectrum of $\mathcal{L}$, which essentially corresponds to all $z\in\mathbb{R}\setminus(\tau/2)\mathbb{Z}$. The reason is deeply rooted in the meromorphic nature of the Harish-Chandra’s c-function (see (\ref{harishchandra})) which has poles at the points $z\in(\tau/2)\mathbb{Z}$. In this regard, we have proved the following result. For the definition of $\mathcal{H}^{r}_{\ast}(\mathcal{X})$ and $\|\cdot\|_{\mathcal{H}^{r}_{\ast}(\mathcal{X})}$, we refer to the expressions (\ref{2hardy1}) and (\ref{2hardy2}).

\begin{Theorem}\label{hardy2case}
Let $z\in\mathbb{R}\setminus(\tau/2)\mathbb{Z}$. Then $u=\mathcal{P}_{z}F$ for some $F\in L^{r}(\Omega)$ if and only if $u\in \mathcal{H}^{r}_{\ast}(\mathcal{X})$ for $1<r<\infty$ and $\mathcal{L}u=\gamma(z)u$.  Moreover, for every $1<r<\infty$, there exist positive constants $C_{1,r}$ and $C_{2,r}$ independent of $z$ such that
$$C_{1,r}|\mathbf{c}(z)|\|F\|_{L^{r}(\Omega)}\leq\|\mathcal{P}_{z}F\|_{\mathcal{H}^{r}_{\ast}(\mathcal{X})}\leq C_{2,r}|\mathbf{c}(z)|\|F\|_{L^{r}(\Omega)},\quad\text{for all }F\in L^{r}(\Omega),$$
where $\mathbf{c}(\cdot)$ is the Harish-Chandra's c-function.
\end{Theorem}

Homogeneous trees are often considered as discrete counterparts of hyperbolic spaces, and more generally noncompact symmetric spaces of rank one (which will be denoted by $X$). In the present context, this analogy has been exploited in both directions in order to exchange ideas and problems from one setting to the other and vice versa. In fact, Theorem \ref{MZ2} can be considered as the tree analogue of the celebrated conjecture by Helgason  \cite{SH1,SH2} which states that every eigenfunction of the Laplace-Beltrami operator $\Delta_{X}$ on $X$ is the Poisson transform of analytic functionals defined on the Furstenberg boundary of $X$ (see also \cite{KKM}). This has set off a flurry of research focusing on refinements of Helgason’s result by imposing various size conditions, ranging from weak $L^{p}$-norms to Hardy-type norms, on the eigenfunctions of $\Delta_{X}$ on $X$. The literature around this area is too numerous to mention. A few representatives are \cite{STN,BS,ADI,KK,K,LR,SP1,SM,S1}.

On the other hand, not much attention was given in characterizing the eigenfunctions of $\mathcal{L}$ that satisfy various size conditions on $\mathcal{X}$. Recently, the problem of characterizing all weak $L^{p}$-eigenfunctions of $\mathcal{L}$ on $\mathcal{X}$ was settled in \cite{KR}. Regarding the Hardy-type eigenfunctions of $\mathcal{L}$ on $\mathcal{X}$, only few results are known which are confined to the case of harmonic functions (see \cite{CCS,KPT,T}). In the context of symmetric spaces, the Poisson boundary representation of harmonic Hardy functions was proved by Stoll \cite{SM}. Later on, Ben Sa\"{\i}d et al. \cite{STN} introduced Hardy-type spaces in a unified manner and a characterization of Hardy-type eigenfunctions corresponding to certain complex eigenvalues of $\Delta_{X}$ (which, in particular includes the harmonic case) were proved on $X$. Theorem \ref{hardypcase} can be considered as the tree analogue of \cite[Theorem 3.6]{STN} proved for symmetric spaces. As in the continuous setup, a key ingredient required to prove Theorem \ref{hardypcase} is the radial convergence result (proved in \cite{KP}) of the normalized Poisson integrals $\phi_{z}(x)^{-1}\mathcal{P}_{z}F(x)$ (where $\phi_{z}$ is the elementary spherical function in $\mathcal{X}$ and $z$ is as in Theorem \ref{hardypcase}), on classical function spaces defined on $\Omega$. However, unlike on symmetric spaces, here one has to get into the realms of the martingale theory on the probability measure space $\Omega$ and derive uniform operator norm estimates of certain one-parameter family of operators $\varepsilon_{n}$ (defined in Section \ref{characterizationpcase}) over function spaces on $\mathcal{X}$.

As we have mentioned earlier, Theorem \ref{hardypcase} leaves out the parameters $z\in\mathbb{R}\setminus(\tau/2)\mathbb{Z}$. This case, in particular, is more interesting and challenging than what we have discussed so far. The reason lies in the fact that for such values of $z$, the spherical functions $\phi_{z}$ oscillates (see (\ref{phiz})), and hence one cannot even define the normalized Poisson integrals $\phi_{z}(x)^{-1}\mathcal{P}_{z}F(x)$. On symmetric spaces of rank one, Strichartz \cite{S1} observed that a similar situation occurs when dealing with the Poisson integrals $P_{\lambda}$ corresponding to all non-zero real parameters $\lambda$. It was also conjectured that for these values of $\lambda$, the eigenfunctions of $\Delta_{X}$ which are Poisson integrals of $L^{2}$-functions on the boundary, are characterized by an $L^{2}$-weight norm on $X$. Ionescu \cite{ADI} proved this conjecture when $X$ is of rank one, which was later extended to higher rank by Kaizuka \cite{KK}. Taking inspiration from Strichartz's conjecture, Boussejra et al. \cite{BS} proved that for all non-zero parameters $\lambda$, the Poisson transform $P_{\lambda}$ associated to a hyperbolic space establishes a bijection between the space of all $L^{r}$-functions defined on the boundary and a suitable subspace of an eigenspace of the Laplace-Beltrami operator and that it preserves certain Hardy-type estimates (see \cite[Theorem B]{BS}). However, as explained in \cite[Remark 2.3]{BS}, only partial results were obtained. We focus on getting its complete extension, hence proving an analogue of \cite[Theorem B]{BS} and settling the conjecture \cite[Remark 2.3]{BS} in the context of homogeneous trees. Theorem \ref{hardy2case} is our endeavor in this pursuit. 

\noindent\textbf{Plan of the paper.} In Section \ref{Section 2} we shall briefly introduce the terminologies and definitions relevant to the subject matter of this article. Here, we shall also define the Hardy spaces on $\mathcal{X}$. Section \ref{characterizationpcase} is devoted to the proof of Theorem \ref{hardypcase}, along with some preparatory results required to prove both of the main results of this paper. Finally, in Section \ref{characterization2case} we shall prove Theorem \ref{hardy2case}.

\noindent\textbf{Generalities.} The letters $\mathbb{N}$, $\mathbb{Z}$, $\mathbb{R}$ and $\mathbb{C}$ will respectively denote the set of natural numbers, ring of integers, field of real and complex numbers. We will denote the set of non-negative integers by $\mathbb{Z}_{+}$. For $z\in \mathbb{C}$, $\Re z$ and $\Im z$ denote respectively the real and imaginary parts of $z$. If $f_{1}$ and $f_{2}$ are two non-negative functions on $X$, we say that $f_{1}\asymp f_{2}$ if there exist positive constants $A$ and $B$ such that $Af_{1}(t)\leq f_{2}(t)\leq Bf_{1}(t)$ for all $t$ in $X$. The letters $C,C_{1},C_{2}$ etc. will be used to denote a finite positive constant whose value may change from one line to another. These constants may be occasionally suffixed to show their dependencies on some important parameters.

\section{Preliminaries}\label{Section 2}
We shall briefly introduce the terminologies and definitions relevant to the subject matter of this article. We shall mainly follow \cite{CS2,FTP1,FTP2}. Our parametrization of various objects such as the Poisson transform and the elementary spherical function will be consistent with that of \cite{CS2} and will differ slightly with that of \cite{FTP1,FTP2}.

\subsection{Homogeneous trees}
A homogeneous tree $\mathcal{X}$ of degree $q+1$ is a connected graph with no cycles where every vertex has $q+1$ neighbours. Throughout this article, we will assume $q\geq 2$. We identify $\mathcal{X}$ with the set of all vertices and endow $\mathcal{X}$ with the canonical discrete distance $d$ defined by the number of edges lying in the unique geodesic path joining the two vertices. We fix an arbitrary reference point $o$ in $\mathcal{X}$ and abbreviate $d(o,x)$ with $|x|$. The tree $\mathcal{X}$ is equipped with the counting measure and we will write $\#E$ to denote the cardinality of a finite subset $E$ of $\mathcal{X}$. For $x\in\mathcal{X}$ and $n\in\mathbb{Z}_{+}$, $B(x,n)$ and $S(x,n)$ will respectively denote the ball and the sphere centered at $x$ and of radius $n$.

Let $G$ be the group of isometries of the metric space $(\mathcal{X},d)$ and $K=\{g\in G:g(o)=o\}$. It is well-known that $G$ acts transitively on $\mathcal{X}$ via the map $(g,x)\mapsto g\cdot x=g(x)$. The action identifies $\mathcal{X}$ with the coset space $G/K$ so that functions on $\mathcal{X}$ identify to right-$K$-invariant functions on $G$ and vice-versa. In particular, the group action induces the following integral formula for a complex-valued function $u$ on $\mathcal{X}$ (see \cite{CMS}):
\begin{equation}\label{formula}
\int\limits_{K}u(k\cdot x)~dk=
\begin{cases}
u(o), & \text{if }x=o,\\
\frac{1}{(q+1)q^{|x|-1}}\sum\limits_{y:|y|=|x|}u(y), & \text{if }x\neq o.
\end{cases}
\end{equation}
Using the above identification, it also follows that radial functions on $\mathcal{X}$, that is, functions which only depend on $|x|$, correspond to $K$-bi-invariant functions on $G$ and vice-versa. If $f:\mathcal{X}\rightarrow\mathbb{C}$ is radial, we shall write $f(n)$ to denote the value of $f$ on $S(o,n)$.

\subsection{The boundary of \texorpdfstring{$\mathcal{X}$}{boundary}}
The boundary $\Omega$ is defined as the set of all infinite geodesic rays, that is, chains of consecutively adjacent vertices, starting from $o$. Once and for all, we fix an arbitrary boundary point $\omega_{o}=\{\omega_{0},\omega_{1},\ldots\}$ and observe that $|\omega_{n}|=n$ for all $n\in\mathbb{Z}_{+}$. For $x\in\mathcal{X}$ and $\omega\in\Omega$, we define $c(x,\omega)=x_{l}$, where $x_{l}$ is the last point lying on $\omega$ in the geodesic path $\{o,x_{1},\ldots,x_{n}\}$ connecting $o$ to $x$. For each $x\in\mathcal{X}$, the sectors
\begin{equation}\label{sectors}
E_{j}(x)=\{\omega\in\Omega:|c(x,\omega)|\geq j\},\quad\text{for all }j\geq 0,
\end{equation}
are compact and form a basis for the topology of $\Omega$. Observe that $E_{0}(x)=\Omega$ and $E_{j}(x)=\emptyset$ for all $j>|x|$. For the sake of simplicity, we use the notation $E(x)$ to abbreviate $E_{|x|}(x)$. Clearly, the collection $\{E(x):x\in S(o,n)\}$ partitions $\Omega$ into $\# S(o,n)$ disjoint open sets. We define the probability measure $\nu$ on these basic open sets by
\begin{equation}\label{measureofsectors}
\nu(E(x))=\begin{cases}
1, & \text{if }x=o,\\
\frac{1}{(q+1)q^{|x|-1}}, & \text{if }x\neq o.
\end{cases}
\end{equation}
Then $(\Omega,\mathcal{M},\nu)$ becomes a probability measure space with the $\sigma$-algebra $\mathcal M$ being generated by the set $\{E(x): x\in \mathcal{X}\}$.  It is easy to see that $K$ acts transitively on $\Omega$. This action leads to the following integral formula for an integrable function $F$ on $\Omega$:
\begin{equation}\label{omegatok}
\int\limits_{K}F(k\cdot\omega_{o})~dk=\int\limits_{\Omega}F(\omega)~d\nu(\omega).
\end{equation}

For each $n\in\mathbb{Z}_{+}$, let $\mathcal M_{n}$ be the sub-algebra of $\mathcal M$ generated by the sets of the form $E(x)$ where $x\in B(o,n)$. It is easy to see that $\{\mathcal M_{n}\}$ is an expanding sequence of $\sigma$-algebras on $\Omega$. In addition, let $\mathcal{K}_{n}(\Omega)$ be a linear space of functions on $\Omega$ which are finite linear combinations of characteristic functions of the sets $E(x)$ where $x$ varies over $B(o,n)$.  Now, define the space
$$\mathcal{K}(\Omega)=\bigcup\limits_{n\geq 0}\mathcal{K}_{n}(\Omega).$$
Since $\mathcal{K}_{n}(\Omega)\subset\mathcal{K}_{n+1}(\Omega)$ for each $n\geq 0$, we find that $F\in\mathcal{K}(\Omega)$ if and only if $F\in\mathcal{K}_{n}(\Omega)$ for some $n$. Functions belonging to this space are often referred to as the cylindrical functions (see, for example \cite[Page 52]{FTP2}). The $n$-th difference $\Delta_{n}(F)$ of a function $F\in\mathcal{K}(\Omega)$ is defined by the formula
\begin{equation}\label{differenceoperators}
\Delta_{n}(F)(\omega)=\mathcal{E}_{n}(F)(\omega)-\mathcal{E}_{n-1}(F)(\omega),
\end{equation}
where $\mathcal{E}_{-1}=0$ and for $n\geq 0$, $\mathcal{E}_{n}(F)$ denotes the $n$-th conditional expectation of $F$ relative to the sub-algebra $\mathcal M_{n}$ which is defined by
\begin{equation}\label{conditionalexpectations}
\mathcal{E}_{n}(F)(\omega)=\frac{1}{\nu(E(x_{n}))}\int\limits_{E(x_{n})}F(\omega^{\prime})~d\nu(\omega^{\prime}),\quad\text{where }\omega=\{x_{j}:j\in\mathbb{Z}_{+}\}\in\Omega.
\end{equation}
It follows from the definition that
\begin{equation}\label{knomegaformula}
F(\omega)=\mathcal{E}_{n}(F)(\omega)=\sum\limits_{j=0}^{n}\Delta_{j}(F)(\omega),\quad\text{whenever }F\in\mathcal{K}_{n}(\Omega)\text{ for some }n\in\mathbb{Z}_{+}.
\end{equation}
A detailed discussion regarding these topics can be found in \cite[Chapter IV]{ST}.

\subsection{The Poisson kernel and the Laplacian}
The Poisson kernel $p(g\cdot o,\omega)$ is defined to be the Radon-Nikodym derivative $d\nu(g^{-1}\cdot\omega)/d\nu(\omega)$. As shown in \cite[Chapter 3, Section 2]{FTP2}, it holds that
\begin{equation}\label{poissonoriginal}
p(x,\omega)=q^{2|c(x,\omega)|-|x|},\quad\text{for all }x\in \mathcal{X},\text{ for all }\omega\in\Omega.
\end{equation}
Since, for a fixed $x\in\mathcal{X}$, $|c(x,\omega)|$ takes only finitely many values as a function of $\omega$, we can decompose the Poisson kernel as
\begin{equation} \label{poissonfull}
p(x,\omega)=\sum\limits_{j=0}^{|x|}q^{2j-|x|}\chi_{E_j(x)\setminus E_{j+1}(x)}(\omega),\quad\text{for all }x\in \mathcal{X},\text{ for all }\omega\in\Omega.
\end{equation}
Using expression (\ref{poissonoriginal}), it also follows that the function $x\mapsto p^{1/2+iz}(x,\omega)$ satisfies $\mathcal{L}p^{1/2+iz}(x,\omega)=\gamma(z)p^{1/2+iz}(x,\omega)$ for all $x\in\mathcal{X}$, where the Laplacian $\mathcal{L}$ is defined by the formula
$$\mathcal{L}u(x)=u(x)-\frac{1}{q+1}\sum\limits_{y:d(x,y)=1}u(y),\quad\text{for all }x\in\mathcal{X},$$
and $\gamma:\mathbb{C}\rightarrow\mathbb{C}$ is an entire function which is defined by
\begin{equation}\label{gammaz}
\gamma(z)=1-\frac{q^{1/2+iz}+q^{1/2-iz}}{q+1}.
\end{equation}

\subsection{The Poisson transform}
For $z\in\mathbb{C}$, the Poisson transform $\mathcal{P}_{z}F$ of a function $F\in\mathcal{K}(\Omega)$ is defined by (see \cite[Chapter 4, Formula (1) at Page 53]{FTP2})
\begin{equation}\label{poissontransform}
\mathcal{P}_{z}F(x)=\int\limits_{\Omega}p^{1/2+iz}(x,\omega)F(\omega)d\nu(\omega),\quad\text{for all }x\in\mathcal{X}.
\end{equation}
Alternatively, the Poisson transform can be viewed as a matrix coefficient of a representation $\pi_z$ of $G$ on $\mathcal{K}(\Omega)$. For details, we refer the interested readers to \cite[Chapter 4]{FTP2}. The notion of Poisson transform can be further extended to elements of the dual space of $\mathcal{K}(\Omega)$, that is, the space of all martingales on the boundary $\Omega$ (see \cite[Page 376]{MZ}). A martingale $\mathbf{F}=\{F_{n}\}_{n\in\mathbb{Z}_{+}}$ is a sequence of measurable functions on $\Omega$ such that $F_{n}\in\mathcal{K}_{n}(\Omega)$ for each $n\geq 0$ and for every $m,n\in\mathbb{Z}_{+}$, $\mathcal{E}_{m}(F_{n})=F_{k}$ where $k=\min\{m,n\}$. It can be shown that each $F\in L^{r}(\Omega)$, $1\leq r\leq\infty$, identifies to a martingale $\mathbf F=\{\mathcal{E}_{n}(F)\}_{n\in\mathbb{Z}_{+}}$ via the conditional expectation. However, the converse is not true.

\begin{Proposition}\label{martingalelpfunction}
Let $1<r\leq\infty$. A martingale $\mathbf{F}=\{F_{n}\}_{n\in\mathbb{Z}_{+}}$ is the conditional expectation of a unique $F\in L^{r}(\Omega)$ if
$$\sup\limits_{n\in\mathbb{Z}_{+}}\|F_{n}\|_{L^{r}(\Omega)}<\infty.$$
Furthermore, if a martingale $\mathbf{F}=\{F_{n}\}_{n\in\mathbb{Z}_{+}}$ satisfies the above condition with $r=1$, then it is the conditional expectation of a unique complex measure $\mu$ on $\Omega$.
\end{Proposition}

The Poisson transform of a martingale  $\mathbf{F}=\{F_{n}\}_{n\in\mathbb{Z}_{+}}$ is defined by the formula
$$\mathcal{P}_{z}\mathbf{F}(x)=\lim\limits_{n\rightarrow\infty}~\int\limits_{\Omega}p^{1/2+iz}(x,\omega)F_{n}(\omega)d\nu(\omega),\quad\text{for all }x\in\mathcal{X}.$$
For each $x$, the above sequence is eventually constant and hence the limit exists. Clearly, $\mathcal{P}_{z+\tau}=\mathcal{P}_{z}$ where $\tau=2\pi/\log q$. It is also a well-known fact that $\mathcal{L}(\mathcal{P}_{z}\mathbf{F})=\gamma(z)\mathcal{P}_{z}\mathbf{F}$, for every $z\in\mathbb{C}$ (see \cite{MZ} for details). 

\subsection{The elementary spherical functions}
For $z\in\mathbb{C}$, the elementary spherical function $\phi_z$ is defined as the Poisson transform of the constant function $\bf{1}$. It follows immediately that $|\phi_{z}(x)|\leq\phi_{i\Im z}(x)$ and $\phi_{i\Im z}(x)>0$ for all $x\in\mathcal{X}$ and for all $z\in\mathbb{C}$. It is also known that $\phi_{z}$ is a radial eigenfunction of $\mathcal{L}$ with eigenvalue $\gamma(z)$ and every radial eigenfunction of $\mathcal{L}$ with eigenvalue $\gamma(z)$ is a constant multiple of $\phi_{z}$ (see \cite[Theorem 1]{FTP1}). The explicit formula for $\phi_{z}$ is as follows (see \cite[Theorem 2]{FTP1} and \cite[Chapter 3, Theorem 2.2]{FTP2}):
\begin{equation}\label{phiz}
\phi_z(n)= \begin{cases}
\vspace*{.2cm} \left(\frac{q-1}{q+1}n+1\right)q^{-n/2}, & \text{when } z\in\ \tau\mathbb{Z},\\
\vspace*{.2cm}\left(\frac{q-1}{q+1}n+1\right)q^{-n/2}(-1)^{n}, & \text{when } z\in {\tau/2}+\tau\mathbb{Z},\\
\mathbf{c}(z)q^{{(iz-1/2)}n}+\mathbf{c}(-z)q^{{(-iz-1/2)}n}, & \text{when } z\in\mathbb{C}\setminus(\tau/2)\mathbb{Z},
\end{cases}
\end{equation}
where $\mathbf{c}$ is a meromorphic function (aka the Harish-Chandra's c-function) given by
\begin{equation}\label{harishchandra}
\mathbf{c}(z)=\frac{q^{1/2}}{q+1}\frac{q^{1/2+iz}-q^{-{1/2}-iz}}{q^{iz}-q^{-iz}},\quad\text{for all }z\in \mathbb{C}\setminus(\tau/2)\mathbb{Z}.
\end{equation}
From the explicit formula (\ref{phiz}), we obtain $\phi_{z}=\phi_{-z}=\phi_{z+\tau}$ for all $z\in\mathbb{C}$. It also follows from (\ref{harishchandra}) that
\begin{equation}\label{harishchandraproperties}
\overline{\mathbf{c}(z)}=\mathbf{c}(-z)\quad\text{and}\quad\mathbf{c}(z)+\mathbf{c}(-z)=1,\quad\text{for all }z\in \mathbb{R}\setminus(\tau/2)\mathbb{Z}.
\end{equation}
We further observe that the c-function has neither zero nor pole in the region $\Im z<0$ (see (\ref{harishchandra})). Implementing these facts into (\ref{phiz}), for all $z=\alpha\pm i\delta_{p}$, $1\leq p<2$, there exists $n_{0}\in\mathbb{N}$ such that
\begin{equation}\label{phipointwise1}
A_{p}~q^{-\frac{n}{p^{\prime}}}\leq |\phi_{z}(n)|\leq B_{p}~q^{-\frac{n}{p^{\prime}}},\quad\text{whenever }n\geq n_{0},
\end{equation}
where $A_{p}$ and $B_{p}$ are positive constants which depend only on $p$. When $p=2$ and $z\in (\tau/2)\Z$, the expression (\ref{phiz}) yields
\begin{equation}\label{phipointwise2}
A~(1+n)~q^{-\frac{n}{2}}\leq |\phi_{z}(n)|\leq B~(1+n)~q^{-\frac{n}{2}},\quad\text{for all }n\in\mathbb{Z}_{+}.
\end{equation}
However, due to the meromorphic nature of the Harish-Chandra's c-function, the pointwise estimate of $\phi_{z}$, when $z\in\mathbb{R}\setminus(\tau/2)\Z$, varies slightly in comparison to that of the $p$ case (see (\ref{phipointwise1})). Indeed, for all $z\in\mathbb{R}\setminus(\tau/2)\Z$,
$$|\phi_{z}(n)|\leq 2|\mathbf{c}(z)|~q^{-\frac{n}{2}},\quad\text{for all }n\in\mathbb{Z}_{+}.$$

\subsection{Hardy-type spaces on \texorpdfstring{$\mathcal{X}$}{hardy}}
We will now define the Hardy-type spaces $\mathcal{H}^{r}_{p}(\mathcal{X})$ on $\mathcal{X}$. Fix $1\leq p\leq 2$. Let $\mathcal{H}^{r}_{p}(\mathcal{X})$ be the space of all complex-valued functions $u$ on $\mathcal{X}$ for which
\begin{equation}\label{hardy1}
\|u\|_{\mathcal{H}^{r}_{p}(\mathcal{X})}:=\sup\limits_{n\in\mathbb{Z}_{+}}~\phi_{i\delta_{p}}(n)^{-1}\left(~\int\limits_{K}|u(k\cdot\omega_{n})|^{r}dk~\right)^{1/r}<\infty,\quad\text{if }1\leq r<\infty,
\end{equation}
and when $r=\infty$,
\begin{equation}\label{hardy2}
\|u\|_{\mathcal{H}^{\infty}_{p}(\mathcal{X})}:=\sup\limits_{x\in\mathcal{X}}~\phi_{i\delta_{p}}(x)^{-1}|u(x)|<\infty.
\end{equation}
Furthermore, we denote by $\mathcal{H}^{r}_{\ast}(\mathcal{X})$ the space of all complex-valued functions $u$ on $\mathcal{X}$ for which
\begin{equation}\label{2hardy1}
\|u\|_{\mathcal{H}^{r}_{\ast}(\mathcal{X})}:=\sup\limits_{n\in\mathbb{Z}_{+}}~q^{n/2}\left(~\int\limits_{K}|u(k\cdot\omega_{n})|^{r}dk~\right)^{1/r}<\infty,\quad\text{if }1\leq r<\infty,
\end{equation}
and when $r=\infty$,
\begin{equation}\label{2hardy2}
\|u\|_{\mathcal{H}^{\infty}_{\ast}(\mathcal{X})}:=\sup\limits_{x\in\mathcal{X}}~q^{|x|/2}|u(x)|<\infty.
\end{equation}

\section{Characterization of \texorpdfstring{$\mathcal{H}^{r}_{p}(\mathcal{X})$}{pcase}-eigenfunctions of \texorpdfstring{$\mathcal{L}$}{laplacian}}\label{characterizationpcase}
Here we shall prove Theorem \ref{hardypcase}. We begin by recalling the following radial convergence result of the normalized Poisson transform proved in \cite{KP}. According to our parametrization of the Poisson transform, the result can be restated as follows:

\begin{Proposition}[{{\cite[Proposition 1]{KP}}}]\label{radialconvergence}
Let $1\leq p\leq 2$. Suppose that $z=\alpha+i\delta_{p^{\prime}}$, where $\alpha$ satisfies either of the following forms:
\begin{enumerate}
\item[(i)] If $1\leq p<2$, $\alpha\in\mathbb{R}$.
\item[(ii)] If $p=2$, $\alpha\in(\tau/2)\mathbb{Z}$.
\end{enumerate}
If $1\leq r<\infty$ and if $F\in L^{r}(\Omega)$, then
$$\lim\limits_{n\rightarrow\infty} \phi_{z}(n)^{-1}\mathcal{P}_{z}F(k\cdot\omega_{n})=F(k\cdot\omega_{o}),\quad\text{in }L^{r}(K)\text{-norm}.$$
Moreover, if $r=\infty$ and if $F\in L^{\infty}(\Omega)$, the convergence above takes place in weak$^{\ast}$-topology of $L^{\infty}(K)$.
\end{Proposition}

We shall now prove a couple of preparatory results. The following proposition is fairly straightforward, but for the sake of completeness, we provide a proof.

\begin{Proposition}\label{hardyifpart}
Let $1\leq p\leq 2$. Suppose that $z=\alpha+i\delta_{p^{\prime}}$, where $\alpha$ satisfies either of the following forms:
\begin{enumerate}
\item[(i)] If $1\leq p<2$, $\alpha\in\mathbb{R}$.
\item[(ii)] If $p=2$, $\alpha\in(\tau/2)\mathbb{Z}$.
\end{enumerate}
Then there exists a positive constant $C_{p}$, depending only on $p$, such that
$$C_{p}\|F\|_{L^{r}(\Omega)}\leq\|\mathcal{P}_{z}F\|_{\mathcal{H}^{r}_{p}(\mathcal{X})}\leq\|F\|_{L^{r}(\Omega)},\quad\text{for all }F\in L^{r}(\Omega),\text{ for all }1\leq r\leq \infty.$$
\end{Proposition}

\begin{proof}
Fix $1\leq p\leq 2$ and let $z$ be as in the hypothesis of this proposition. Expressing the Poisson transform (\ref{poissontransform}) as an integral over $K$ using (\ref{omegatok}) and then plugging in the formula $p(k\cdot \omega_{n},k_{1}\cdot\omega_{o})=p(\omega_{n},k^{-1}k_{1}\cdot\omega_{o})$ of the Poisson kernel (see \cite[Page 285]{FTP1}), we obtain
$$\mathcal{P}_{z}F(k\cdot\omega_{n})=\int\limits_{K}p^{1/2+iz}(\omega_{n},k^{-1}k_{1}\cdot\omega_{o})~F(k_{1}\cdot\omega_{o})~dk_{1},\quad\text{for all }F\in \mathcal{K}(\Omega).$$
Setting $p_{1}(\omega_{n},k_{2}\cdot\omega_{o})=p(\omega_{n},k^{-1}_{2}\cdot\omega_{o})$ for all $k_{2}\in K$, the integral above takes the form
$$\mathcal{P}_{z}F(k\cdot\omega_{n})=\left(F\ast p^{1/2+iz}_{1}(\omega_{n},\cdot)\right)(k\cdot\omega_{o}).$$
Now taking modulus on both the sides of the expression above and using the Minkowski's inequality, we obtain
$$\|\mathcal{P}_{z}F(\cdot\cdot\omega_{n})\|_{L^{r}(K)}\leq \|p^{1/p}_{1}(\omega_{n},\cdot)\|_{L^{1}(K)}\|F\|_{L^{r}(K)},\quad\text{for all }1\leq r\leq \infty.$$
Since $\|p^{1/p}_{1}(\omega_{n},\cdot)\|_{L^{1}(K)}=\phi_{i\delta_{p^{\prime}}}(n)=\phi_{i\delta_{p}}(n)$ for all $n$, we get
\begin{equation}\label{pcasefirstinequality}
\|\mathcal{P}_{z}F\|_{\mathcal{H}^{r}_{p}(\mathcal{X})}\leq\|F\|_{L^{r}(\Omega)},\quad\text{for all }F\in L^{r}(\Omega),\text{ for all }1\leq r\leq \infty.
\end{equation}

Conversely, if $p$ and $z$ satisfies the hypothesis of this proposition then from (\ref{pcasefirstinequality}), we have $\mathcal{P}_{z}F\in\mathcal{H}^{r}_{p}(\mathcal{X})$ for all $F\in L^{r}(\Omega)$, where $1\leq r\leq\infty$. Next, we fix $F\in L^{r}(\Omega)$ for some $1\leq r<\infty$. Using Proposition \ref{radialconvergence}, we get a strictly increasing subsequence $\{n_{m}\}$ of natural numbers such that 
$$\lim\limits_{n_{m}\rightarrow\infty}\phi_{z}(n_{m})^{-1}\mathcal{P}_{z}F(k\cdot\omega_{n_{m}})=F(k\cdot\omega_{o}),\quad\text{for almost every }k\in K.$$
Using Fatou's lemma and the fact that $|\phi_{z}(n)|\geq A_{p}~|\phi_{i\delta_{p}}(n)|$ for all $n$ large enough (see (\ref{phipointwise1}) and (\ref{phipointwise2})), we finally have
$$\|F\|_{L^{r}(\Omega)}\leq B_{p}~\|\mathcal{P}_{z}F\|_{\mathcal{H}^{r}_{p}(\mathcal{X})},\quad\text{for all }F\in L^{r}(\Omega),\text{ for all }1\leq r\leq \infty,$$
as claimed. We next turn to the case $r=\infty$. Once again using Proposition \ref{radialconvergence}, we obtain
$$\lim\limits_{n\rightarrow\infty}\int\limits_{K}G(k\cdot\omega_{o})~\frac{\mathcal{P}_{z}F(k\cdot\omega_{n})}{\phi_{z}(n)}~dk=\int\limits_{K}G(k\cdot\omega_{o})~F(k\cdot\omega_{o})~dk,\quad\text{for all }G\in L^{1}(\Omega).$$
Since for all $n\geq 0$,
$$\left|~\int\limits_{K}G(k\cdot\omega_{o})~\frac{\mathcal{P}_{z}F(k\cdot\omega_{n})}{\phi_{z}(n)}~dk~\right|\leq A_{p}~\|G\|_{L^{1}(K)}\|\mathcal{P}_{z}F\|_{\mathcal{H}^{\infty}_{p}(\mathcal{X})},$$
using the duality argument, we finally get $C_{p}~\|F\|_{L^{\infty}(\Omega)}\leq \|\mathcal{P}_{z}F\|_{\mathcal{H}^{\infty}_{p}(\mathcal{X})}$, which completes the proof.
\end{proof}

To proceed further, we need to study the uniform Hardy-type boundedness properties of the operators $\varepsilon_{n}$, $n\in\mathbb{Z}_{+}$, introduced in \cite{MZ}. Given a complex-valued function $u$ on $\mathcal{X}$, we define $\varepsilon_{n}u$ by the rule (see \cite[Page 378]{MZ})
\begin{equation}\label{epsilonn}
\varepsilon_{n}u(x)=\frac{1}{\#\mathscr{S}(n,x)}\sum\limits_{y\in\mathscr{S}(n,x)}u(y),\quad\text{for all }x\in\mathcal{X},
\end{equation}
where
$$\mathscr{S}(n,x)= \begin{cases}
\vspace*{.2cm} \{x\}&\text{if }|x|\leq n,\\
\vspace*{.2cm} \{y\in\mathfrak{X}: |y|=|x|, x_{n}=y_{n}\}&\text{if }|x|>n,\\
\end{cases}$$
and $\{o=x_{0},\ldots,x_{n},\ldots,x_{n+k}=x\}$ and $\{o=y_{0},\ldots,y_{n},\ldots,y_{n+k}=y\}$ are geodesic rays connecting $o$ to $x$ and $o$ to $y$ respectively. The operators $\varepsilon_{n}$ are essentially the analogue of the conditional expectations $\mathcal{E}_{n}$ defined in Section \ref{Section 2} for functions on $\Omega$.

\begin{Lemma}\label{epsilonnboundedness}
Let $1\leq p\leq 2$ and $1\leq r\leq\infty$. For every $n\in\mathbb{Z}_{+}$, the operators $\varepsilon_{n}$ defined by (\ref{epsilonn}) are bounded from $\mathcal{H}^{r}_{p}(\mathcal{X})$ to itself and $\mathcal{H}^{r}_{\ast}(\mathcal{X})$ to itself. In particular, for all $n\in\mathbb{Z}_{+}$,
\begin{align*}
\|\varepsilon_{n}u\|_{\mathcal{H}^{r}_{p}(\mathcal{X})}&\leq\|u\|_{\mathcal{H}^{r}_{p}(\mathcal{X})},\quad\text{for all }u\in \mathcal{H}^{r}_{p}(\mathcal{X}),\\
\|\varepsilon_{n}u\|_{\mathcal{H}^{r}_{\ast}(\mathcal{X})}&\leq \|u\|_{\mathcal{H}^{r}_{\ast}(\mathcal{X})},\quad\text{for all }u\in \mathcal{H}^{r}_{\ast}(\mathcal{X}).
\end{align*}
\end{Lemma}

\begin{proof}
Consider the measure space $(S(o,m),\mathscr{A}_{m},\#)$, where $S(o,m)$ is the geodesic sphere of radius $m$ around $o$ in $\mathcal{X}$, $\mathscr{A}_{m}$ is the $\sigma$-algebra generated by the set $\{\{x\}:x\in S(o,m)\}$ and $\#$ denotes the counting measure. Our first aim is to prove that $\varepsilon_{n}$ is bounded from $L^{r}(S(o,m),\mathscr{A}_{m},\#)$ to itself, for every $n,m\in\mathbb{Z}_{+}$ and for all $1\leq r\leq\infty$. Fix $n\in\mathbb{Z}_{+}$. If $m\leq n$, $\varepsilon_{n}u(x)=u(x)$ for all $x\in S(o,m)$ and therefore
\begin{equation}\label{mlessthann}
\|\varepsilon_{n}u\|_{L^{r}(S(o,m))}=\|u\|_{L^{r}(S(o,m))},\quad\text{for all }1\leq r\leq\infty.
\end{equation}
We next assume that $m>n$. In such cases, $\varepsilon_{n}$ acts as a projection operator from the space of all $\mathscr{A}_{m}$-measurable functions on $S(o,m)$ onto the space of all $\mathscr{A}_{m,n}$-measurable functions on $S(o,m)$, where $\mathscr{A}_{m,n}$ is a sub $\sigma$-algebra of $\mathscr{A}_{m}$ which is generated by the set $\{\mathscr{S}(n,y):y\in S(o,m)\}$. Moreover, it is easy to see that $\varepsilon_{n}$ satisfies the following properties:
\begin{enumerate}
\item[(a)] $\varepsilon_{n}u=u$, if $u$ is $\mathscr{A}_{m,n}$-measurable on $S(o,m)$.
\item[(b)] For any two $\mathscr{A}_{m}$-measurable functions $u$ and $v$ on $S(o,m)$, we have that
$$\sum\limits_{y\in S(o,m)}\varepsilon_{n}u(y)~v(y)=\sum\limits_{y\in S(o,m)}u(y)~\varepsilon_{n}v(y).$$
\end{enumerate}
Property (a) follows from the definition. To prove Property (b), for each $x_{n}\in S(o,n)$, we define $S^{\prime}(x_{n},m)=\{y\in S(o,m): x_{n}=y_{n}\}$, where $\{o=y_{0},\ldots,y_{n},\ldots,y_{m}=y\}$ is the geodesic ray connecting $o$ to $y$. Now observe that $S(o,m)$ is the disjoint union of the sets $S^{\prime}(x_{n},m)$, $x_{n}\in S(o,n)$ and therefore
$$\sum\limits_{y\in S(o,m)}\varepsilon_{n}u(y)~v(y)=\sum\limits_{x_{n}\in S(o,n)}\left(~\sum\limits_{y\in S^{\prime}(x_{n},m)}\varepsilon_{n}u(y)~v(y)~\right).$$
Since $S^{\prime}(x_{n},m)=\mathscr{S}(n,y)$ for every $y\in S^{\prime}(x_{n},m)$ and $\varepsilon_{n}u$ is constant on $S^{\prime}(x_{n},m)$, we finally have
$$\sum\limits_{y\in S(o,m)}\varepsilon_{n}u(y)~v(y)=\sum\limits_{x_{n}\in S(o,n)}\# \mathscr{S}(n,x)~\varepsilon_{n}u(x)~\varepsilon_{n}v(x),\quad\text{for some }x\in S^{\prime}(x_{n},m).$$
Interchanging the role of $u$ and $v$, we again arrive at the same expression, which establishes our claim.

Using the above properties of $\varepsilon_{n}$, it follows that for all $1\leq r\leq\infty$ and for all $m>n$,
\begin{align}
\|\varepsilon_{n}u\|_{L^{r}(S(o,m))}&=\sup\left\{\left|\sum\limits_{y\in S(o,m)}\varepsilon_{n}u(y)~v(y)\right|:v\text{ is }\mathscr{A}_{m,n}\text{-measurable and }\|v\|_{L^{r^{\prime}}(S(o,m))}\leq 1\right\}\nonumber\\
&=\sup\left\{\left|\sum\limits_{y\in S(o,m)}u(y)~\varepsilon_{n}v(y)\right|:v\text{ is }\mathscr{A}_{m,n}\text{-measurable and }\|v\|_{L^{r^{\prime}}(S(o,m))}\leq 1\right\}\nonumber\\
&=\sup\left\{\left|\sum\limits_{y\in S(o,m)}u(y)~v(y)\right|:v\text{ is }\mathscr{A}_{m,n}\text{-measurable and }\|v\|_{L^{r^{\prime}}(S(o,m))}\leq 1\right\}\nonumber\\
&\leq\sup\left\{\left|\sum\limits_{y\in S(o,m)}u(y)~v(y)\right|:v\text{ is }\mathscr{A}_{m}\text{-measurable and }\|v\|_{L^{r^{\prime}}(S(o,m))}\leq 1\right\}\nonumber\\
&=\|u\|_{L^{r}(S(o,m))}.\label{mgreaterthann}
\end{align}
The desired result now follows by combining (\ref{mlessthann}), (\ref{mgreaterthann}) and (\ref{formula}).
\end{proof}

\begin{Remark}
\textbf{(1)} It follows from (\ref{mlessthann}) and (\ref{mgreaterthann}) that the operators $\varepsilon_{n}$, $n\in\mathbb{Z}_{+}$, are bounded from $L^{1}(S(o,m),\mathscr{A}_{m},\#)$ to itself and their corresponding operator norms are independent of $m$. However, the associated maximal operator $\varepsilon^{\ast}$ defined by
$$\varepsilon^{\ast}u(x)=\sup\limits_{n\in\mathbb{Z}_{+}}|\varepsilon_{n}u(x)|,\quad\text{for all }x\in\mathcal{X},$$
is bounded from $L^{1}(S(o,m),\mathscr{A}_{m},\#)$ to $L^{1,\infty}(S(o,m),\mathscr{A}_{m},\#)$ uniformly in $m$. Indeed, we first observe that for a fixed $m\in\mathbb{Z}_{+}$,
\begin{equation}\label{maximalfunction}
\varepsilon^{\ast}u(x)=\sup\limits_{0\leq n\leq m}|\varepsilon_{n}u(x)|,\quad\text{for all }x\in S(o,m).
\end{equation}
Now, let us assume, without the loss of generality that $u$ and hence all $\varepsilon_{n}u$ are non-negative functions on $S(o,m)$. For $\lambda>0$, let $S_{\lambda,m}=\{x\in S(o,m):\varepsilon^{\ast}u(x)>\lambda\}$. Then $S_{\lambda,m}$ can be decomposed into the disjoint union of the sets $S_{\lambda,m,n}$, $n=0,\ldots,m$, where
$$S_{\lambda,m,n}=\{x\in S(o,m):\varepsilon_{n}u(x)>\lambda\text{ but }\varepsilon_{k}u(x)\leq\lambda\text{ for all }k=0,\ldots,n-1\}.$$
It is fairly easy to see that $S_{\lambda,m,n}$ is a $\mathscr{A}_{m,n}$-measurable set for each $n$. From this fact, expression (\ref{maximalfunction}) and properties (a), (b) of $\varepsilon_{n}$ (see the proof of Lemma \ref{epsilonnboundedness}) we deduce that
\begin{align}
\sum\limits_{x\in S_{\lambda,m}}u(x)=\sum\limits_{n=0}^{m}~\left(\sum\limits_{x\in S(o,m)}u(x)~\varepsilon_{n}(\chi_{S_{\lambda,m,n}})(x)\right)&=\sum\limits_{n=0}^{m}~\left(\sum\limits_{x\in S(o,m)}\varepsilon_{n}u(x)~\chi_{S_{\lambda,m,n}}(x)\right)\nonumber\\
&=\sum\limits_{n=0}^{m}~\left(\sum\limits_{x\in S_{\lambda,m,n}}\varepsilon_{n}u(x)\right)\geq \lambda~(\# S_{\lambda,m}),\label{maximalweaktype}
\end{align}
where $\chi_{S_{\lambda,m,n}}(\cdot)$ denotes the indicator function of the set $S_{\lambda,m,n}$. From (\ref{maximalweaktype}), we finally deduce that $\varepsilon^{\ast}$ is bounded from $L^{1}(S(o,m),\mathscr{A}_{m},\#)$ to $L^{1,\infty}(S(o,m),\mathscr{A}_{m},\#)$ and the operator norm is independent of $m$, as claimed.

\noindent\textbf{(2)} From the preceding remark, one can actually prove that the operator $\varepsilon^{\ast}$ is of weak type $(1,1)$ on $L^{1}(\mathcal{X})$ and of strong type $(p,p)$ on $L^{p}(\mathcal{X})$ for all $1<p\leq\infty$. To see this, let us consider the set $S_{\lambda}=\{x\in\mathcal{X}:\varepsilon^{\ast}u(x)>\lambda\}$, where $\lambda>0$. Then estimate (\ref{maximalweaktype}) implies that
$$\lambda~(\# S_{\lambda})=\sum\limits_{m=0}^{\infty} \lambda~(\# S_{\lambda,m})\leq \|u\|_{L^{1}(\mathfrak{X})}.$$
The $L^{p}$-boundedness now follows by interpolating the weak type $(1,1)$-boundedness and the trivial $L^{\infty}$-boundedness of $\varepsilon^{\ast}$. Therefore, in comparison to the previously obtained results in \cite[Page 733]{KR}, here we get an improvement in the boundedness properties of $\varepsilon^{\ast}$.
\end{Remark}

Since we have developed all the necessary ingredients, we are now ready to conclude the proof of Theorem \ref{hardypcase}.

\noindent\textit{Proof of Theorem \ref{hardypcase}.} Assume that $u=\mathcal{P}_{z}F$ for some $F\in L^{r}(\Omega)$, where $r$ and $z$ are as in the hypothesis of this theorem. Since, for a fixed $\omega$, the function $x\mapsto p^{1/2+iz}(x,\omega)$ is an eigenfunction of $\mathcal{L}$ with the eigenvalue $\gamma(z)$, we have $\mathcal{L}u=\gamma(z)u$. Furthermore, using Proposition \ref{hardyifpart} it follows that $u\in\mathcal{H}^{r}_{p}(\mathcal{X})$ and there exists a constant $C_{p}>0$ such that
$$C_{p}\|F\|_{L^{r}(\Omega)}\leq\|\mathcal{P}_{z}F\|_{\mathcal{H}^{r}_{p}(\mathcal{X})}\leq\|F\|_{L^{r}(\Omega)},\quad\text{for all }F\in L^{r}(\Omega),\text{ for all }1\leq r\leq \infty.$$
Furthermore, replacing the function $F$ by a complex measure $\mu$ on $\Omega$, a similar computation as in Proposition \ref{hardyifpart} yields
$$\|\mathcal{P}_{z}\mu\|_{\mathcal{H}^{1}_{p}(\mathcal{X})}\leq|\mu|(\Omega), \text{ where }|\mu|\text{ denotes the total variation of }\mu\text{ on }\Omega.$$

Conversely, let $\mathcal{L}u=\gamma(z)u$ for some $z=\alpha+i\delta_{p^{\prime}}$, where $1\leq p\leq 2$ and $\alpha$ satisfies the hypothesis of this theorem. By Theorem \ref{MZ2}, there exists a martingale $\mathbf{F}=\{F_{n}\}_{n\in\mathbb{Z}_{+}}$ on the boundary $\Omega$ such that
$$u(x)=\mathcal{P}_{z}\mathbf{F}(x)=\lim\limits_{n\rightarrow\infty}\int\limits_{\Omega}p^{1/2+iz}(x,\omega)F_{n}(\omega)~d\nu(\omega).$$ 
Since $F_{n}\in\mathcal{K}(\Omega)\subset L^{r}(\Omega)$ for all $1\leq r\leq\infty$, using Proposition \ref{hardyifpart}, there exists a constant $C_{p}>0$, depending only on $p$, such that
$$\|F_{n}\|_{L^{r}(\Omega)}\leq C_{p}\|\mathcal{P}_{z}F_{n}\|_{\mathcal{H}^{r}_{p}(\mathcal{X})},\quad\text{for every }n\in\mathbb{Z}_{+}.$$
Now, plugging in the formula $\mathcal{P}_{z}F_{n}=\varepsilon_{n}(\mathcal{P}_{z}\mathbf{F})$ (see \cite[Lemma 3.3]{MZ}) above, we further get
$$\|F_{n}\|_{L^{r}(\Omega)}\leq C_{p}\|\varepsilon_{n}u\|_{\mathcal{H}^{r}_{p}(\mathcal{X})},\quad\text{for every }n\in\mathbb{Z}_{+},$$
where $\varepsilon_{n}$ is as in (\ref{epsilonn}). The hypothesis $u\in\mathcal{H}^{r}_{p}(\mathcal{X})$, Lemma \ref{epsilonnboundedness} and  the expression above altogether implies that
$$\sup\limits_{n\in\mathbb{Z}_{+}}\|F_{n}\|_{L^{r}(\Omega)}\leq C_{p}\|u\|_{\mathcal{H}^{r}_{p}(\mathcal{X})}<\infty.$$
Finally, using Proposition \ref{martingalelpfunction}, we get our desired result. \qed

\section{Characterization of \texorpdfstring{$\mathcal{H}^{r}_{\ast}(\mathcal{X})$}{2case}-eigenfunctions of \texorpdfstring{$\mathcal{L}$}{laplacian}}\label{characterization2case}
To obtain an analogue of the radial convergence result, that is Proposition \ref{radialconvergence}, for all $z\in\mathbb{R}\setminus(\tau/2)\mathbb{Z}$ and to prove Theorem \ref{hardy2case}, we will use the boundedness properties of the martingale maximal function and a martingale version of the Littlewood-Paley inequalities. We state these results in the form of lemmas and refer to \cite[Chapter IV]{ST} for the proofs.

\begin{Lemma}[{{\cite[Chapter IV, Theorem 6]{ST}}}]\label{martingalemaximalfunction}
For $F\in\mathcal{K}(\Omega)$ and $n\in\mathbb{N}$, let $\mathcal{E}_{n}(F)$ be as in (\ref{conditionalexpectations}). Define the associated maximal operator by the formula
$$\mathcal{E}^{\ast}(F)(\omega)=\sup\limits_{n\in\mathbb{N}}|\mathcal{E}_{n}(F)(\omega)|,\quad\text{for all }F\in\mathcal{K}(\Omega).$$
Then for every $1<r\leq\infty$, there exists a positive constant $C_{r}$ depending only on $r$ such that
$$\|\mathcal{E}^{\ast}(F)\|_{L^{r}(\Omega)}\leq C_{r}\|F\|_{L^{r}(\Omega)},\quad\text{for all }F\in L^{r}(\Omega).$$
\end{Lemma}

\begin{Lemma}[{{\cite[Chapter IV, Theorem 7]{ST}}}]\label{differenceoperatorboundedness}
For $F\in\mathcal{K}(\Omega)$ and $j\in\mathbb{N}$, let $\Delta_{j}(F)$ be as in (\ref{differenceoperators}). Suppose that $\mathbf{a}=\{a_{j}\}_{j\in\mathbb{N}}$ is any sequence of numbers such that $|a_{j}|\leq 1$ for all $j\in\mathbb{N}$. Set
$$T_{\mathbf{a}}(F)(\omega)=\sum\limits_{j=1}^{\infty}a_{j}~\Delta_{j}(F)(\omega),\quad\text{for all }F\in\mathcal{K}(\Omega).$$
Then for every $1<r<\infty$, there exists a positive constant $C_{r}$ independent of $\mathbf{a}$ such that
$$\|T_{\mathbf{a}}(F)\|_{L^{r}(\Omega)}\leq C_{r}\|F\|_{L^{r}(\Omega)},\quad\text{for all }F\in L^{r}(\Omega).$$
\end{Lemma}

From the explicit expression (\ref{phiz}), it is clear that for $z\in\mathbb{R}\setminus(\tau/2)\mathbb{Z}$, the spherical functions $\phi_{z}$ oscillates, and hence one cannot expect to get a radial convergence result of the normalized Poisson integral $\phi_{z}(|x|)^{-1}\mathcal{P}_{z}F(x)$, as in Proposition \ref{radialconvergence}. Therefore to avoid the zeros of $\phi_{z}$, we instead consider the normalized form $q^{|x|/2}\mathcal{P}_{z}F(x)$, whenever $z\in\mathbb{R}\setminus(\tau/2)\mathbb{Z}$. Using (\ref{phiz}), it is also evident that $q^{|x|/2}\mathcal{P}_{z}\mathbf{1}(x)$ satisfies the following convergence:
$$\lim\limits_{n\rightarrow\infty}\left(\frac{1}{n}\sum\limits_{l=0}^{n}\left(q^{l/2}\phi_{z}(l)\right)~\left(q^{l/2}\overline{\phi_{z}(l)}\right)\right)=2|\mathbf{c}(z)|^{2}.$$
On the other hand, if we substitute the facts $\#S(o,l)\asymp q^{l}$ and $\overline{\phi_{z}(l)}=\phi_{-z}(l)$ inside the quantity in the left hand side of the expression above and simplify it further using (\ref{formula}), we get
$$\frac{1}{n}\sum\limits_{l=0}^{n}q^{l}\phi_{z}(l)\overline{\phi_{z}(l)}\asymp \frac{1}{n}\sum\limits_{l=0}^{n}\#S(o,l)\phi_{z}(l)\phi_{-z}(l)=\frac{1}{n}\sum\limits_{x\in B(o,n)}p^{1/2-iz}(x,\omega)\mathcal{P}_{z}\mathbf{1}(x).$$
Considering this as our first line of motivation, we will now prove the following convergence result:

\begin{Proposition}\label{inversionpoisson}
Let $z\in\mathbb{R}\setminus(\tau/2)\mathbb{Z}$. If $1<r<\infty$ and if $F\in L^{r}(\Omega)$, then
$$\frac{q}{2(q+1)|\mathbf{c}(z)|^{2}}~\lim\limits_{n\rightarrow\infty}\left(\frac{1}{n}\sum\limits_{x\in B(o,n)}p^{1/2-iz}(x,\omega)\mathcal{P}_{z}F(x)\right)=F(\omega),\quad\text{in }L^{r}(\Omega)\text{-norm}.$$
\end{Proposition}

\begin{proof}
Let us assume that $z\in\mathbb{R}\setminus(\tau/2)\mathbb{Z}$. Our first aim is to prove our assertion for all $F\in\mathcal{K}(\Omega)$. For this purpose, we define a one parameter family of operators $\{T_{n}\}_{n\in\mathbb{N}}$ by the formula
\begin{equation}\label{partialsum2case}
T_{n}F(\omega)=\frac{1}{n}\sum\limits_{x\in B(o,n)}p^{1/2-iz}(x,\omega)~\mathcal{P}_{z}F(x),\quad\text{for all }\omega\in\Omega,\text{ for all }F\in\mathcal{K}(\Omega).
\end{equation}
Now fix $F\in\mathcal{K}(\Omega)$. By definition, $F\in\mathcal{K}(\Omega)$ implies that $F\in\mathcal{K}_{m}(\Omega)$ for some $m\in\mathbb{Z}_{+}$. Therefore plugging in the formula (\ref{knomegaformula}) into (\ref{partialsum2case}) and simplifying the expression further using (\ref{formula}), we obtain
\begin{equation}\label{inversion2}
T_{n}F(\omega)=\frac{1}{n}~\sum\limits_{l=0}^{n}\# S(o,l)\left(~\sum\limits_{j=0}^{m}~\left(~\int\limits_{K}p^{1/2-iz}(k\cdot\omega_{l},\omega)~\mathcal{P}_{z}(\Delta_{j}(F))(k\cdot\omega_{l})~dk~\right)~\right).
\end{equation}

Our next aim is to explicitly determine $T_{n}F(\omega)$ appearing in the formula above. For this purpose, we shall use the following explicit formula of $\mathcal{P}_{z}(\Delta_{j}(F))(k\cdot\omega_{l})$ from \cite[Page 377]{MZ}:
\begin{equation}\label{poissondeltaj}
\mathcal{P}_{z}(\Delta_{j}(F))(k\cdot\omega_{l})=\begin{cases}
0&\mbox{if }l<j,\\
q^{-l/2}B(j,l,z)~\Delta_{j}(F)(k\cdot\omega_{o})&\mbox{if }l\geq j,
\end{cases}
\end{equation}
where $B(j,l,z)$ is defined by the rule
\begin{equation}\label{bjmz}
B(j,l,z)=\begin{cases}
{\bf c}(z)q^{izl}+{\bf c}(-z)q^{-izl}&\mbox{if }j=0,\\
{\bf c}(z)q^{iz(j-1)}(q^{iz(l-j+1)}-q^{-iz(l-j+1)})&\mbox{if }j>0.
\end{cases}
\end{equation}
Substituting the value of $\mathcal{P}_{z}(\Delta_{j}(F))(k\cdot\omega_{l})$ from (\ref{poissondeltaj}) into (\ref{inversion2}), for all $n\geq m+1$, we have
\begin{align}
T_{n}F(\omega)&=\frac{1}{n}~\sum\limits_{l=0}^{m}\# S(o,l)q^{-l/2}\left(~\sum\limits_{j=0}^{l}B(j,l,z)\left(~\int\limits_{K}p^{1/2-iz}(k\cdot\omega_{l},\omega)~\Delta_{j}(F)(k\cdot\omega_{o})~dk~\right)~\right)\nonumber\\
&+\frac{1}{n}~\sum\limits_{l=m+1}^{n}\# S(o,l)q^{-l/2}\left(~\sum\limits_{j=0}^{m}B(j,l,z)\left(~\int\limits_{K}p^{1/2-iz}(k\cdot\omega_{l},\omega)~\Delta_{j}(F)(k\cdot\omega_{o})~dk~\right)~\right).\label{gnomega1}
\end{align}

Since $k\cdot o=o$, using (\ref{poissonoriginal}) we have $p^{1/2-iz}(k\cdot o,k_{1}\cdot\omega_{o})=1$ for all $k\in K$ and for a fixed $k_{1}\in K$. This in turn implies that
\begin{equation}\label{poisson1}
\int\limits_{K}p^{1/2-iz}(k\cdot\omega_{l},k_{1}\cdot\omega_{o})~\Delta_{j}(F)(k\cdot\omega_{o})~dk=\mathcal{P}_{-z}(\Delta_{j}(F))(k_{1}\cdot\omega_{l}),\quad\text{whenever }l=0.
\end{equation}
We next compute the integral in (\ref{poisson1}) for all $l\in\mathbb{N}$. By definition, for a fixed $k_{1}\in K$ and $l\in\mathbb{N}$, the function $k\mapsto |c(k\cdot\omega_{l},k_{1}\cdot\omega_{o})|$ takes only finitely many values. Implementing this fact in (\ref{poissonoriginal}), we get a decomposition of the form
\begin{multline*}
p^{1/2-iz}(k\cdot\omega_{l},k_{1}\cdot\omega_{o})=\sum\limits_{d=0}^{l-1}q^{(1/2-iz)(2d-l)}\chi_{\{k\in K:k\cdot\omega_{d}=k_{1}\cdot\omega_{d}\}\setminus\{k\in K:k\cdot\omega_{d+1}=k_{1}\cdot\omega_{d+1}\}}(k)\\
+q^{(1/2-iz)l}\chi_{\{k\in K:k\cdot\omega_{l}=k_{1}\cdot\omega_{l}\}}(k).
\end{multline*}
Now observe that for each $d=0,\ldots,l$, the set $\{k\in K:k\cdot\omega_{d}=k_{1}\cdot\omega_{d}\}$ corresponds to the sector $E_{d}(k_{1}\cdot\omega_{l})$ defined by (\ref{sectors}), via the group action $k\mapsto k\cdot\omega_{o}$. Keeping this in mind and comparing the decomposition above with the explicit expression (\ref{poissonfull}) of the Poisson kernel, it follows that $p^{1/2-iz}(k\cdot\omega_{l},k_{1}\cdot\omega_{o})=p^{1/2-iz}(k_{1}\cdot\omega_{l},k\cdot\omega_{o})$. Consequently,
\begin{equation}\label{poisson2}
\int\limits_{K}p^{1/2-iz}(k\cdot\omega_{l},k_{1}\cdot\omega_{o})~\Delta_{j}(F)(k\cdot\omega_{o})~dk=\mathcal{P}_{-z}(\Delta_{j}(F))(k_{1}\cdot\omega_{l}),\quad\text{for all }l\in\mathbb{N}.
\end{equation}
Plugging in the values of the integrals from (\ref{poisson1}) and (\ref{poisson2}) into (\ref{gnomega1}), we get, for all $n\geq m+1$,
\begin{align}
T_{n}F(k_{1}\cdot\omega_{o})&=\frac{1}{n}~\sum\limits_{l=0}^{m}\# S(o,l)q^{-l/2}\left(~\sum\limits_{j=0}^{l}B(j,l,z)~\mathcal{P}_{-z}(\Delta_{j}(F))(k_{1}\cdot\omega_{l})~\right)\nonumber\\
&\hspace*{7.5em}+\frac{1}{n}~\sum\limits_{l=m+1}^{n}\# S(o,l)q^{-l/2}\left(~\sum\limits_{j=0}^{m}B(j,l,z)\mathcal{P}_{-z}(\Delta_{j}(F))(k_{1}\cdot\omega_{l})~\right).\label{gnomega2}
\end{align}
Now substituting the value of $\mathcal{P}_{-z}(\Delta_{j}(F))(k_{1}\cdot\omega_{l})$ from (\ref{poissondeltaj}) into (\ref{gnomega2}) and noting that $B(j,l,-z)=\overline{B(j,l,z)}$ (see (\ref{bjmz}) and (\ref{harishchandraproperties})), we obtain, for all $n\geq m+1$,
\begin{multline}\label{prefinal}
T_{n}F(\omega)=\frac{1}{n}~\sum\limits_{l=0}^{m}\# S(o,l)q^{-l}\left(~\sum\limits_{j=0}^{l}|B(j,l,z)|^{2}~\Delta_{j}(F)(\omega)~\right)\\
+\frac{1}{n}~\sum\limits_{l=m+1}^{n}\# S(o,l)q^{-l}\left(~\sum\limits_{j=0}^{m}|B(j,l,z)|^{2}~\Delta_{j}(F)(\omega)~\right).
\end{multline}
Implementing a change of summation in (\ref{prefinal}) and imitating all the calculations above for $n\leq m$, it follows that for all $F\in\mathcal{K}_{m}(\Omega)$ and for all $n\in\mathbb{N}$,
\begin{equation}\label{gnfinal}
T_{n}F(\omega)=\sum\limits_{j=0}^{\min\{m,n\}}\left(~\frac{1}{n}~\sum\limits_{l=j}^{n}\# S(o,l)q^{-l}~|B(j,l,z)|^{2}~\right)\Delta_{j}(F)(\omega).
\end{equation}
Moreover, a simple computation using the explicit expression of $B(j,l,z)$ from (\ref{bjmz}) and the properties of the c-function from (\ref{harishchandraproperties}) gives us the following:
\begin{multline}\label{finalcoefficients}
\frac{1}{n}~\sum\limits_{l=j}^{n}\# S(o,l)q^{-l}~|B(j,l,z)|^{2}=\\
\begin{cases}
\frac{1}{n}~\left[~1+\left(\frac{q+1}{q}\right)2n|\mathbf{c}(z)|^{2}+\left(\frac{1-q^{2izn}}{1-q^{2iz}}\right)\mathbf{c}(z)^{2}q^{2iz}+\left(\frac{1-q^{-2izn}}{1-q^{-2iz}}\right)\mathbf{c}(-z)^{2}q^{-2iz}~\right], & \text{ if }j=0,\\
 & \\
\frac{q+1}{q}~\frac{|\mathbf{c}(z)|^{2}}{n}~\left[~2\left(n-j+1\right)-\left(\frac{1-q^{2iz(n-j+1)}}{1-q^{2iz}}\right)q^{2iz}-\left(\frac{1-q^{-2iz(n-j+1)}}{1-q^{-2iz}}\right)q^{-2iz}~\right], & \text{ if }j>0.
\end{cases}
\end{multline}

Therefore expanding $F\in\mathcal{K}_{m}(\Omega)$ using (\ref{knomegaformula}) and then using (\ref{gnfinal}), (\ref{finalcoefficients}) and the fact that $\Delta_{0}(F)=\mathcal{E}_{0}(F)$, we obtain
\begin{equation}\label{finalequation}
\frac{q}{2(q+1)|\mathbf{c}(z)|^{2}}T_{n}F(\omega)-F(\omega)=a_{0,n}~\mathcal{E}_{0}(F)(\omega)+\sum\limits_{j=1}^{m}a_{j,n}~\Delta_{j}(F)(\omega),\quad\text{for all }n\geq m+1,
\end{equation}
where, for a fixed $n$, the complex numbers $a_{j,n}$ are given by
\begin{multline}\label{finalcoefficients2}
a_{j,n}=
\begin{cases}
\frac{q}{2n(q+1)|\mathbf{c}(z)|^{2}}~\left[~1+\left(\frac{1-q^{2izn}}{1-q^{2iz}}\right)\mathbf{c}(z)^{2}q^{2iz}+\left(\frac{1-q^{-2izn}}{1-q^{-2iz}}\right)\mathbf{c}(-z)^{2}q^{-2iz}~\right], & \text{if }j=0,\\
 & \\
\frac{1}{2n}~\left[~2\left(1-j\right)-\left(\frac{1-q^{2iz(n-j+1)}}{1-q^{2iz}}\right)q^{2iz}-\left(\frac{1-q^{-2iz(n-j+1)}}{1-q^{-2iz}}\right)q^{-2iz}~\right], & \text{if }j>0.
\end{cases}
\end{multline}
Now observe that if we set $b_{m,z}=2m+4|1-q^{2iz}|^{-1}$, then one can write
\begin{equation}\label{sumdeltaj}
\sum\limits_{j=1}^{m}a_{j,n}~\Delta_{j}(F)(\omega)=\frac{b_{m,z}}{2n}\left(~\sum\limits_{j=1}^{m}\frac{2n~a_{j,n}}{b_{m,z}}~\Delta_{j}(F)(\omega)~\right)=\frac{b_{m,z}}{2n}~T_{\mathbf{a^{\prime}_{n}}}(F)(\omega),
\end{equation}
where $T_{\mathbf{a^{\prime}_{n}}}(F)$ is as in Lemma \ref{differenceoperatorboundedness}, $\mathbf{a^{\prime}_{n}}=\{a^{\prime}_{j,n}\}_{j\in\mathbb{N}}$ and
$$a^{\prime}_{j,n}=\begin{cases}
\frac{2n~a_{j,n}}{b_{m,z}}, & \text{if }j\leq m,\\
0, & \text{if }j\geq m+1.
\end{cases}$$
Furthermore, from (\ref{finalcoefficients2}), It is easy to see that $|a^{\prime}_{j,n}|=(2n~|a_{j,n}|)/|b_{m,z}|\leq1$, for all $j\in\mathbb{N}$. Therefore, using Lemma \ref{differenceoperatorboundedness} in the right hand side of the expression (\ref{sumdeltaj}), for every $1<r<\infty$, we get a positive constant $C_{r}$ independent of $\mathbf{a^{\prime}_{n}}$ such that
\begin{equation}\label{convergence1}
\|\sum\limits_{j=1}^{m}a_{j,n}~\Delta_{j}(F)\|_{L^{r}(\Omega)}\leq \frac{b_{m,z}}{2n}~C_{r}\|F\|_{L^{r}(\Omega)},\quad\text{for all }F\in\mathcal{K}_{m}(\Omega),\text{ for all }n\geq m+1.
\end{equation}
Finally, taking $L^{r}(\Omega)$-norm on both sides of (\ref{finalequation}), using (\ref{convergence1}) and the trivial boundedness $\|\mathcal{E}_{0}(F)\|_{L^{r}(\Omega)}\leq \|F\|_{L^{r}(\Omega)}$, we have, for all $F\in\mathcal{K}_{m}(\Omega)$ and $1<r<\infty$,
\begin{equation}\label{assertiondenseclass}
\left\|~\frac{q}{2(q+1)|\mathbf{c}(z)|^{2}}T_{n}F-F~\right\|_{L^{r}(\Omega)}\leq \frac{C_{m,z,r}}{2n}~\|F\|_{L^{r}(\Omega)}\rightarrow 0,\quad\text{as }n\rightarrow\infty.
\end{equation}

We are now ready to prove our assertion for functions in $L^{r}(\Omega)$, where $1<r<\infty$. However, before proceeding further, we first observe from (\ref{finalcoefficients}) that for a fixed $n\in\mathbb{N}$ and for every $m\in\mathbb{Z}_{+}$, there exists a positive constant $C_{z}$ depending only on $z$ such that
$$\frac{1}{n}~\sum\limits_{l=j}^{n}\# S(o,l)q^{-l}~|B(j,l,z)|^{2}\leq C_{z},\quad\text{for all }0\leq j\leq\min\{m,n\}.$$
Implementing this fact in (\ref{gnfinal}) and then using Lemma \ref{differenceoperatorboundedness}, we get, for $1<r<\infty$,
\begin{equation}\label{uniformbound}
\|T_{n}F\|_{L^{r}(\Omega)}\leq C_{z,r}\|F\|_{L^{r}(\Omega)},\quad\text{for all }F\in\mathcal{K}(\Omega),\text{ for all }n\in\mathbb{N}.
\end{equation}
Consequently, using the density argument, it is possible to extend $T_{n}$ uniquely as a bounded linear operator on $L^{r}(\Omega)$ so that (\ref{uniformbound}) holds for all $F\in L^{r}(\Omega)$ and for all $n\in\mathbb{N}$. Now fix $\epsilon>0$. By density, for a given $F\in L^{r}(\Omega)$, there exists some $G\in \mathcal{K}_{m}(\Omega)$ such that
\begin{equation}\label{density1}
\|G-F\|_{L^{r}(\Omega)}<\epsilon^{\prime}<\min\left\{\frac{\epsilon}{3},\frac{2\epsilon(q+1)|\mathbf{c}(z)|^{2}}{3C_{z,r}~q}\right\}.
\end{equation}
By the result proved earlier (see (\ref{assertiondenseclass})), we have
\begin{equation}\label{density2}
\left\|~\frac{q}{2(q+1)|\mathbf{c}(z)|^{2}}T_{n}G-G~\right\|_{L^{r}(\Omega)}<\frac{\epsilon}{3},\quad\text{for all }n\geq n_{0}\geq m+1.
\end{equation}
Finally, using (\ref{uniformbound}), (\ref{density1}) and (\ref{density2}), we have in conclusion,
\begin{align*}
&\left\|~\frac{q}{2(q+1)|\mathbf{c}(z)|^{2}}T_{n}F-F~\right\|_{L^{r}(\Omega)}\\
&\leq \frac{q}{2(q+1)|\mathbf{c}(z)|^{2}}\|T_{n}F-T_{n}G\|_{L^{r}(\Omega)}+\left\|~\frac{q}{2(q+1)|\mathbf{c}(z)|^{2}}T_{n}G-G~\right\|_{L^{r}(\Omega)}+~\|G-F\|_{L^{r}(\Omega)}\\
&\leq\frac{C_{z,r}~q}{2(q+1)|\mathbf{c}(z)|^{2}}\|F-G\|_{L^{r}(\Omega)}+\left\|~\frac{q}{2(q+1)|\mathbf{c}(z)|^{2}}T_{n}G-G~\right\|_{L^{r}(\Omega)}+~\|G-F\|_{L^{r}(\Omega)}\\
&<\frac{\epsilon}{3}+\frac{\epsilon}{3}+\frac{\epsilon}{3}=\epsilon,\quad\text{for all }n\geq n_{0},
\end{align*}
which is our claim.
\end{proof}

\begin{Corollary}
Let $z\in\mathbb{R}\setminus(\tau/2)\mathbb{Z}$. The characterization result \cite[Theorem A]{KR} states that every $u\in L^{2,\infty}(\mathcal{X})$ satisfying $\mathcal{L}u=\gamma(z)u$ can be written as $u=\mathcal{P}_{z}F$ for a unique $F\in L^{2}(\Omega)$. Combining this fact with Proposition \ref{inversionpoisson}, we have the following: If $u\in L^{2,\infty}(\mathcal{X})$ satisfies $\mathcal{L}u=\gamma(z)u$, then there exists a unique $F\in L^{2}(\Omega)$ such that
$$\frac{q}{2(q+1)|\mathbf{c}(z)|^{2}}~\lim\limits_{n\rightarrow\infty}\left(\frac{1}{n}\sum\limits_{x\in B(o,n)}p^{1/2-iz}(x,\omega)u(x)\right)=F(\omega),\quad\text{in }L^{2}(\Omega)\text{-norm}.$$
\end{Corollary}

We are now ready to prove Theorem \ref{hardy2case}. However, before we start, we observe that for $z\in\mathbb{R}\setminus(\tau/2)\mathbb{Z}$, one cannot use the convolution technique as in Proposition \ref{hardyifpart} to get the $\mathcal{H}^{r}_{\ast}(\mathcal{X})$-estimate of $\mathcal{P}_{z}F$. This is mostly due to the presence of certain oscillatory factors in the expression of $\mathcal{P}_{z}F$. Therefore our main aim is to obtain a pointwise estimate of $\mathcal{P}_{z}F$ by following a similar strategy as in \cite[Theorem 4.1]{SP2}. On the other hand, in order to prove the characterization part of Theorem \ref{hardy2case}, we shall use Proposition \ref{inversionpoisson}.

\vspace*{0.1in}

\noindent\textit{Proof of Theorem \ref{hardy2case}.} Fix $z\in\mathbb{R}\setminus(\tau/2)\mathbb{Z}$ and assume $F\in\mathcal{K}(\Omega)$. Clearly, $u=\mathcal{P}_{z}F$ is an eigenfunction of $\mathcal{L}$ with eigenvalue $\gamma(z)$. It remains to prove that $u\in \mathcal{H}^{r}_{\ast}(\mathcal{X})$ and that there exists a positive constant $C_{r}$ independent of $z$ such that
\begin{equation}\label{2caseifpart}
\|\mathcal{P}_{z}F\|_{\mathcal{H}^{r}_{\ast}(\mathcal{X})}\leq C_{r}|\mathbf{c}(z)|\|F\|_{L^{r}(\Omega)},\quad\text{for all }F\in\mathcal{K}(\Omega),\text{ for all }1<r<\infty.
\end{equation}
For this purpose, we now estimate $\mathcal{P}_{z}F$ pointwise. Fix $x=k\cdot\omega_{n}\in\mathcal{X}$. Substituting the formula (\ref{poissonfull}) of $\omega\mapsto p(x,\omega)$ into (\ref{poissontransform}) and further using (\ref{measureofsectors}), we obtain
$$\mathcal{P}_{z}F(k\cdot\omega_{n})=q^{-(1/2+iz)n}\left(\mathcal{E}_{0}(F)(k\cdot\omega_{o})+\left(\frac{q-q^{-2iz}}{q+1}\right)\sum\limits_{j=1}^{n}q^{2izj}~\mathcal{E}_{j}(F)(k\cdot\omega_{o})\right).$$
Now plugging in the formula (\ref{knomegaformula}) of $\mathcal{E}_{j}(F)$ into the expression above, we get
$$\mathcal{P}_{z}F(k\cdot\omega_{n})=q^{-(1/2+iz)n}\left(\mathcal{E}_{0}(F)(k\cdot\omega_{o})+\left(\frac{q-q^{-2iz}}{q+1}\right)\sum\limits_{j=1}^{n}q^{2izj}\left(~\sum\limits_{l=0}^{j}\Delta_{l}(F)(k\cdot\omega_{o})~\right)\right),$$
which after a change of summation yields
\begin{multline}\label{changeofsummationstep}
\mathcal{P}_{z}F(k\cdot\omega_{n})=q^{-(1/2+iz)n}\left(\left(1+\frac{q^{1+2iz}-1}{(q+1)(1-q^{2iz})}\right)\mathcal{E}_{0}(F)(k\cdot\omega_{o})\right.\\
-\left.\left(\frac{q-q^{-2iz}}{(q+1)(1-q^{2iz})}\right)\left(~q^{2iz(n+1)}\mathcal{E}_{n}(F)(k\cdot\omega_{o})-\sum\limits_{l=1}^{n}q^{2izl}~\Delta_{l}(F)(k\cdot\omega_{o})~\right)\right).
\end{multline}
Simplifying (\ref{changeofsummationstep}) further using the expression (\ref{harishchandra}) and formulae (\ref{harishchandraproperties}) of the Harish-Chandra's c-function, we get
\begin{multline*}
\mathcal{P}_{z}F(k\cdot\omega_{n})=q^{-(1/2+iz)n}\left(\overline{\mathbf{c}(z)}~\mathcal{E}_{0}(F)(k\cdot\omega_{o})+\mathbf{c}(z)q^{2izn}~\mathcal{E}_{n}(F)(k\cdot\omega_{o})\right.\\
-\left.\mathbf{c}(z)q^{-2iz}\left(~\sum\limits_{l=1}^{n}q^{2izl}~\Delta_{l}(F)(k\cdot\omega_{o})~\right)\right).
\end{multline*}
Taking modulus on both sides of the expression above, for all $n\in\mathbb{N}$ and for all $k\in K$, we finally have
\begin{align}
|\mathcal{P}_{z}F(k\cdot\omega_{n})|&\leq q^{-n/2}|\mathbf{c}(z)|\Big(~|\mathcal{E}_{0}(F)(k\cdot\omega_{o})|+|\mathcal{E}_{n}(F)(k\cdot\omega_{o})|+\left|\mathcal{E}_{n}(T_{\mathbf{a}}(F))(k\cdot\omega_{o})\right|~\Big)\nonumber\\
&\leq q^{-n/2}|\mathbf{c}(z)|\Big(~|\mathcal{E}_{0}(F)(k\cdot\omega_{o})|+\mathcal{E}^{\ast}(F)(k\cdot\omega_{o})+\mathcal{E}^{\ast}(T_{\mathbf{a}}(F))(k\cdot\omega_{o})~\Big)\label{2hardyright},
\end{align}
where $\mathbf{a}=\{q^{2izj}\}_{j\in\mathbb{N}}$ and $\mathcal{E}^{\ast}(F)$, $T_{\mathbf{a}}(F)$ are as in Lemma \ref{martingalemaximalfunction} and \ref{differenceoperatorboundedness} respectively. Computing the $L^{r}(K)$-norm on both sides of (\ref{2hardyright}), using Lemma \ref{martingalemaximalfunction}, \ref{differenceoperatorboundedness} and the trivial boundedness $\|\mathcal{E}_{0}(F)\|_{L^{r}(\Omega)}\leq \|F\|_{L^{r}(\Omega)}$, the required assertion (\ref{2caseifpart}) follows.

Conversely, let $\mathcal{L}u=\gamma(z)u$ for some $z\in\mathbb{R}\setminus(\tau/2)\mathbb{Z}$. By Theorem \ref{MZ2}, there exists a martingale $\mathbf{F}=\{F_{n}\}_{n\in\mathbb{Z}_{+}}$ on the boundary $\Omega$ such that $u=\mathcal{P}_{z}\mathbf{F}$. Since $u\in \mathcal{H}^{r}_{\ast}(\mathcal{X})$ and $\varepsilon_{n}u=\mathcal{P}_{z}F_{n}$ (see \cite[Lemma 3.3]{MZ}), using Lemma \ref{epsilonnboundedness} we get
$$\|\mathcal{P}_{z}F_{n}\|_{\mathcal{H}^{r}_{\ast}(\mathcal{X})}\leq \|u\|_{\mathcal{H}^{r}_{\ast}(\mathcal{X})},\quad\text{for all }n\in\mathbb{Z}_{+}.$$
Therefore, in view of Proposition \ref{martingalelpfunction}, we only need to show that for every $1<r<\infty$, there exists a positive constant $C_{r}$ independent of $z$ such that
\begin{equation}\label{2caseonlyif}
C_{r}|\mathbf{c}(z)|\|F\|_{L^{r}(\Omega)}\leq\|\mathcal{P}_{z}F\|_{\mathcal{H}^{r}_{\ast}(\mathcal{X})},\quad\text{for all }F\in\mathcal{K}(\Omega).
\end{equation}

Fix $F\in\mathcal{K}(\Omega)$. Let $T_{n}F$ be defined by the formula (\ref{inversion2}). Using Proposition \ref{inversionpoisson}, it follows that
\begin{equation}\label{weakconvergence}
\lim\limits_{n\rightarrow\infty}\int\limits_{\Omega}T_{n}F(\omega)~G(\omega)~d\nu(\omega)=\frac{2(q+1)|\mathbf{c}(z)|^{2}}{q}\int\limits_{\Omega}F(\omega)~G(\omega)~d\nu(\omega),\quad\text{for all }G\in\mathcal{K}(\Omega).
\end{equation}
Plugging in the value of $T_{n}F(\omega)$ from (\ref{inversion2}) into the left hand integral of the expression above and simplifying the expression further using (\ref{formula}), we obtain, for every $n\in\mathbb{N}$,
$$\int\limits_{K}T_{n}F(\omega)~G(\omega)~d\nu(\omega)=\frac{1}{n}~\sum\limits_{l=0}^{n}\# S(o,l)\left(~\int\limits_{K}\mathcal{P}_{z}F(k\cdot\omega_{l})~\mathcal{P}_{-z}G(k\cdot\omega_{l})~dk~\right).$$
Now taking modulus on both sides of the expression above, using H\"older's inequality and the fact $\# S(o,l)\leq q^{l}$, it follows that
\begin{multline*}
\left|~\int\limits_{K}T_{n}F(\omega)G(\omega)d\nu(\omega)\right|\leq\frac{1}{n}~\sum\limits_{l=0}^{n}q^{l}\left(~\int\limits_{K}|\mathcal{P}_{z}F(k\cdot\omega_{l})|^{r}dk\right)^{\frac{1}{r}}\left(~\int\limits_{K}|\mathcal{P}_{-z}G(k\cdot\omega_{l})|^{r^{\prime}}dk\right)^{\frac{1}{r^{\prime}}}.
\end{multline*}
Since $G\in L^{r^{\prime}}(\Omega)$ for $1<r^{\prime}<\infty$, using (\ref{2caseifpart}) and the fact that $|\mathbf{c}(z)|=|\mathbf{c}(-z)|$ (see (\ref{harishchandraproperties})), we get
\begin{align*}
\left|~\int\limits_{K}T_{n}F(\omega)~G(\omega)~d\nu(\omega)~\right|&\leq\frac{C_{r}|\mathbf{c}(z)|\|G\|_{L^{r^{\prime}}(\Omega)}}{n}\sum\limits_{l=0}^{n}q^{\frac{l}{2}}\left(~\int\limits_{K}|\mathcal{P}_{z}F(k\cdot\omega_{l})|^{r}dk\right)^{\frac{1}{r}}\\
&\leq C_{r}|\mathbf{c}(z)|\|G\|_{L^{r^{\prime}}(\Omega)}\|\mathcal{P}_{z}F\|_{\mathcal{H}^{r}_{\ast}(\mathcal{X})}.
\end{align*}
Combining the inequality above with (\ref{weakconvergence}) and then applying standard duality argument, we get (\ref{2caseonlyif}) which was required to be proved. \qed

\section*{Acknowledgement}
The author wishes to thank Rudra P. Sarkar for carefully reading the manuscript and suggesting several important changes which improved an earlier draft. The author also gratefully acknowledges the support provided by the National Board of Higher Mathematics (NBHM) post-doctoral fellowship (Number: 0204/3/2021/R$\&$D-II/7386) from the Department of Atomic Energy (DAE), Government of India.

\end{document}